\documentclass[11pt]{article}
\usepackage{amsmath, amssymb, amsthm}
\usepackage{booktabs}
\usepackage{geometry}
\usepackage{hyperref}
\usepackage{graphicx}
\usepackage{bm}
\newtheorem{theorem}{Theorem}[section]
\newtheorem{lemma}[theorem]{Lemma}
\newtheorem{proposition}[theorem]{Proposition}
\newtheorem{corollary}[theorem]{Corollary}
\newtheorem{definition}[theorem]{Definition}
\newtheorem{example}[theorem]{Example}
\newtheorem{remark}[theorem]{Remark}
\newtheorem{question}{Question}

\newcommand{\W}{\mathcal{W}}

\title{Cliques in graphs constructed from Strongly Orthogonal Subsets in exceptional root systems}
\author{Patrick J. Browne\footnote{Technical University of the Shannon (TUS). The author gratefully acknowledges support from the Centre for Pedagogical Innovation and Development (CPID) via SATLE funding at TUS.} \and Pádraig Ó Catháin\footnote{Údarás na Gaeltachta, Na Forbacha, Co. na Gaillimhe.}}
\date{April 2026}

\begin{document}
\maketitle

\begin{abstract}
Given a root system $R$, two roots are said to be \emph{strongly orthogonal} if neither their sum nor difference is a root. Gashi defined a family of graphs with vertices labelled by sums of $k$-element strongly orthogonal subsets of roots, and edges connect vertices whose difference is also a vertex. Gashi and the current authors established Erd\H{o}s--Ko--Rado type results for graphs developed from Type $A$ root systems. 

In this paper, we study graphs developed from the exceptional root systems $G_2$, $F_4$, $E_6$, $E_7$, and $E_8$. We compute graph-theoretic invariants including regularity, connectivity, and clique numbers, and analyze clique structures with respect to sunflower properties. The automorphism group contains the Weyl group; we use these symmetries to obtain complete counts of maximum cliques and maximum sunflowers. Unlike type $A$, where all maximal cliques are sunflowers for large rank, sunflower cliques comprise at most 11\% of maximum cliques in the simply-laced exceptional types $E_6$, $E_7$, and $E_8$.
\end{abstract}



\section{Introduction}

The Erd\H{o}s--Ko--Rado (EKR) theorem is a cornerstone of extremal combinatorics, \cite{EKR}. It bounds the size of \emph{intersecting families}: collections $\mathcal{F}$ of $k$-subsets of an $n$-set such that any two members have nonempty intersection. When $n \geq 2k$, the theorem states that $|\mathcal{F}| \leq \binom{n-1}{k-1}$, with equality if and only if $\mathcal{F}$ consists of all $k$-subsets containing a fixed element. In the EKR literature, extremal configurations where all pairwise intersections contain a common set are called \emph{sunflowers}. 

Browne, Gashi, and \'O Cath\'ain developed an EKR-type result for root systems, \cite{BGO}. Given a root system $R$, they constructed graphs $\Gamma(R, k)$ whose vertices are sums of $k$-element strongly orthogonal subsets of roots, with edges when the difference of two vertices is itself a vertex. Their work focused on the type $A$ root system, establishing that for sufficiently large rank, all maximal cliques in $\Gamma(A_\ell, k)$ are sunflowers.

The present paper extends the investigation of EKR properties to the five exceptional type root systems: $G_2$, $F_4$, $E_6$, $E_7$, and $E_8$. We compute the standard graph parameters: numbers of edges, vertices, vertex degrees and sizes of connected components. We also prove that $\W(R) \leq \mathrm{Aut}(\Gamma(R, k))$ in all cases, where $\W(R)$ is the Weyl group of the root system. We compute the numbers of cliques and sunflower cliques, completing the EKR analysis. We observe that certain pairs of graphs are isomorphic, and identify an explicit isomorphism: under conditions that we specify precisely, the sum of $4t$ strongly orthogonal roots is equal to a sum of $t$ strongly orthogonal roots scaled by a factor of $2$. 

Section~\ref{sec:rootsystems} provides background on root systems, with complete descriptions of the exceptional cases. Section~\ref{sec:construction} defines the graphs $\Gamma(R, k)$ and presents their basic parameters. Section~\ref{sec:ekr} recalls the EKR framework developed previously Gashi and the authors, and analyzes clique structures in exceptional types. Section~\ref{sec:conclusion} discusses some open questions and gives remarks on computational methodology.

\section{Root Systems}\label{sec:rootsystems}

Work of Agaoka and Kaneda on symmetric spaces inspired Gashi to investigate combinatorial properties of root systems, \cite{Agaoka} . We give some basic properties of root systems; a standard reference is Chapter~9 of Humphreys' monograph, \cite{Humphreys}.

\begin{definition}\label{RootDefn}
Let $V$ be a finite-dimensional Euclidean space with inner product $\langle \cdot, \cdot \rangle$.
A \emph{root system} in $V$ is a finite set $R \subset V$ satisfying:
\begin{enumerate}
    \item $R$ spans $V$ and $0 \notin R$.
    \item If $\alpha \in R$, then $-\alpha \in R$ and no other scalar multiple of $\alpha$ lies in $R$.
    \item For all $\alpha, \beta \in R$, the reflection $s_\alpha(\beta) = \beta - 2\frac{\langle \beta, \alpha \rangle}{\langle \alpha, \alpha \rangle}\alpha$ lies in $R$.
    \item For all $\alpha, \beta \in R$, we have $2\frac{\langle \beta, \alpha \rangle}{\langle \alpha, \alpha \rangle} \in \mathbb{Z}$.
\end{enumerate}
The \emph{rank} of $R$ is $\dim V$. The \emph{Weyl group}, $\W(R)$, is the subgroup of $\mathrm{GL}(V)$ generated by the reflections $\{s_\alpha : \alpha \in R\}$. It is a finite group that permutes $R$ and acts linearly on $V$.
\end{definition}

A root system is \emph{irreducible} if it cannot be partitioned into two proper subsets such that each root in one set is orthogonal to every root in the other. Irreducible root systems are classified into four infinite families ($A_\ell$, $B_\ell$, $C_\ell$, $D_\ell$) and five exceptional types ($G_2$, $F_4$, $E_6$, $E_7$, $E_8$). An irreducible root system has at most two root lengths; an irreducible root system with a single root length is called \emph{simply laced}. Among the exceptional types, $G_2$ and $F_4$ are non-simply laced.

The \emph{Coxeter number} $h$ of an irreducible root system $R$ of rank $\ell$ is the order of a Coxeter element (the product of a set of base reflections in any order, i.e. reflections through a set of simple roots); it satisfies $|R| = \ell h$. The \emph{dual Coxeter number} $h^\vee$ is defined analogously using the coroot system $R^\vee = \{2\alpha/\|\alpha\|^2 : \alpha \in R\}$; for simply-laced systems, $h = h^\vee$.

A basic property of root systems is the following: if $\alpha$ and $\beta$ are roots with $\langle \alpha, \beta \rangle < 0$ and $\alpha \neq -\beta$, then $\alpha + \beta$ is a root. Equivalently, if $\langle \alpha, \beta \rangle > 0$ and $\alpha \neq \beta$, then $\alpha - \beta$ is a root. This follows from Axiom 4 of Definition \ref{RootDefn}.

\begin{definition}
Two roots $\alpha, \beta \in R$ are \emph{strongly orthogonal} if neither $\alpha + \beta$ nor $\alpha - \beta$ is a root. A \emph{strongly orthogonal subset} (SOS) is a set of pairwise strongly orthogonal roots.
\end{definition}

In the simply-laced case, orthogonality and strong orthogonality coincide: if $\langle \alpha, \beta \rangle = 0$ then $\|\alpha \pm \beta\|^2 > \|\alpha\|^{2}$, so neither $\alpha + \beta$ nor $\alpha - \beta$ is a root. As well as featuring in the work of Agaoka and Kaneda on symmetric spaces, strongly orthogonal roots were used by Burns and Pfeiffer to construct large abelian subgroups of Weyl groups, \cite{BurnsPfeiffer}. 

\begin{proposition}[\cite{Agaoka}]
In a root system of rank $\ell$, a SOS has size at most $\ell$.
\end{proposition}

This bound is achieved by all root systems of exceptional type, except $E_6$ which has maximum SOS size 4. 

\subsection{The Exceptional Root Systems}

We now describe each exceptional root system explicitly, since the sunflower property we describe will depend on the chosen basis. Let $\{e_1, e_2, \ldots, e_n\}$ denote the standard orthonormal basis of $\mathbb{R}^n$.

\boldsymbol{$G_2$} is the smallest exceptional system, with 12 roots in a 2-dimensional space. It is not simply laced. We realize it in $\mathbb{R}^3$ on the hyperplane $x_1 + x_2 + x_3 = 0$.
Roots satisfying $\|w\|^{2} = 2$ are called short roots; they are $\pm(e_i - e_j)$ for $1 \leq i < j \leq 3$. Roots satisfying $\|w\|^{2} = 6$ are called long roots; they are $\pm(2e_i - e_j - e_k)$ for $\{i,j,k\} = \{1,2,3\}$. The Weyl group $\W(G_2)$ is the dihedral group of order 12; the Coxeter number is $h = 6$. The maximum SOS size is 2.

\boldsymbol{$F_4$} is not simply laced, with 48 roots in $\mathbb{R}^4$. The 24 long roots are $\pm e_i \pm e_j$ for $1 \leq i < j \leq 4$, of norm 2. The 24 short roots are the 8 roots $\pm e_i$ for $1 \leq i \leq 4$ and the 16 roots $\frac{1}{2}(\pm e_1 \pm e_2 \pm e_3 \pm e_4)$, of norm $1$. The Weyl group $\W(F_4)$ has order 1152; the Coxeter number is $h = 12$. The maximum SOS size is 4. Note that since $e_1 - e_2$ is a long root, $e_1$ and $e_2$ are orthogonal but \textbf{not} strongly orthogonal. A maximal SOS consists of long roots, e.g. $\{e_1 + e_2,\ e_1 - e_2,\ e_3 + e_4,\ e_3 - e_4\}$. Sums and differences like $(e_1 + e_2) + (e_1 - e_2) = 2e_1$ have norm 4, hence are not roots.

\boldsymbol{$E_8$} is the largest exceptional system, with 240 roots in $\mathbb{R}^8$: the 112 roots $\pm e_i \pm e_j$ for $1 \leq i < j \leq 8$, and the 128 roots $\frac{1}{2}\sum_{i=1}^{8} \epsilon_i e_i$ where each $\epsilon_i = \pm 1$ and $\prod_{i=1}^{8} \epsilon_i = +1$ (even number of minus signs). All roots have norm $2$; this system is simply laced.

The Weyl group $\W(E_8)$ has order $696{,}729{,}600 = 2^{14} \cdot 3^5 \cdot 5^2 \cdot 7$; the Coxeter number is $h = 30$. The maximum SOS size is 8; for example:
\[
\{e_1 + e_2,\ e_1 - e_2,\ e_3 + e_4,\ e_3 - e_4,\ e_5 + e_6,\ e_5 - e_6,\ e_7 + e_8,\ e_7 - e_8\}.
\]

\boldsymbol{$E_7$} has 126 roots in a 7-dimensional subspace of $\mathbb{R}^8$. It consists of all $E_8$ roots orthogonal to a fixed root, conventionally $e_1 + e_8$:
\[
E_7 = \{\alpha \in E_8 : \langle \alpha, e_1 + e_8 \rangle = 0\}.
\]
The Weyl group $\W(E_7)$ has order $2{,}903{,}040 = 2^{10} \cdot 3^4 \cdot 5 \cdot 7$; the Coxeter number is $h = 18$. The maximum SOS size is 7.

\boldsymbol{$E_6$} has 72 roots in a 6-dimensional subspace of $\mathbb{R}^8$. It consists of the $E_8$ roots orthogonal to two roots $\alpha, \beta$ with $\langle \alpha, \beta \rangle = 1$. Taking $\alpha = e_1 + e_7$ and $\beta = e_1 + e_8$:
\[
E_6 = \{\gamma \in E_8 : \langle \gamma, e_1 + e_7 \rangle = \langle \gamma, e_1 + e_8 \rangle = 0\}.
\]
(Removing two orthogonal roots yields a $D_6$ subsystem of 60 roots.)

Explicitly: 40 roots are of the form $\pm e_i \pm e_j$ for $2 \leq i < j \leq 6$; and 32 roots are of the form $\frac{1}{2}\sum_{i=1}^{8} \epsilon_i e_i$ where $\epsilon_1 = -\epsilon_7 = -\epsilon_8$ and the total number of minus signs is even. The Weyl group $\W(E_6)$ has order $51{,}840 = 2^7 \cdot 3^4 \cdot 5$; the Coxeter number is $h = 12$. $E_6$ has an outer automorphism (from the Dynkin diagram symmetry) that doubles this to $103{,}680$.

The maximum SOS size in $E_6$ is 4, strictly less than the rank, as established by Agaoka and Kaneda, \cite{Agaoka}. A maximal SOS is:
$\{e_2 + e_3,\ e_2 - e_3,\ e_4 + e_5,\ e_4 - e_5\}$.

\section{The Graphs $\Gamma(R, k)$}\label{sec:construction}

\begin{definition}\label{Graphdef}
Let $R$ be a root system, and let $\mathrm{SOS}(R, k)$ be the set of all $k$-element strongly orthogonal subsets of $R$. The graph $\Gamma(R, k)$ has:
\begin{itemize}
    \item \textbf{Vertex set:} $\mathcal{V}(R, k) = \left\{\sum_{\alpha \in S} \alpha : S \in \mathrm{SOS}(R, k)\right\}$, the set of sums of $k$-element SOS.
    \item \textbf{Edge set:} $v_1 \sim v_2$ if and only if $v_1 - v_2 \in \mathcal{V}(R, k)$.
\end{itemize}
\end{definition}

Note that multiple SOS may yield the same sum, so $|\mathcal{V}(R,k)| \leq |\mathrm{SOS}(R,k)|$ in general. The next example gives a fully developed description of an SOS graph. 

\begin{figure}[ht]
\centering
\includegraphics[width=0.7\textwidth]{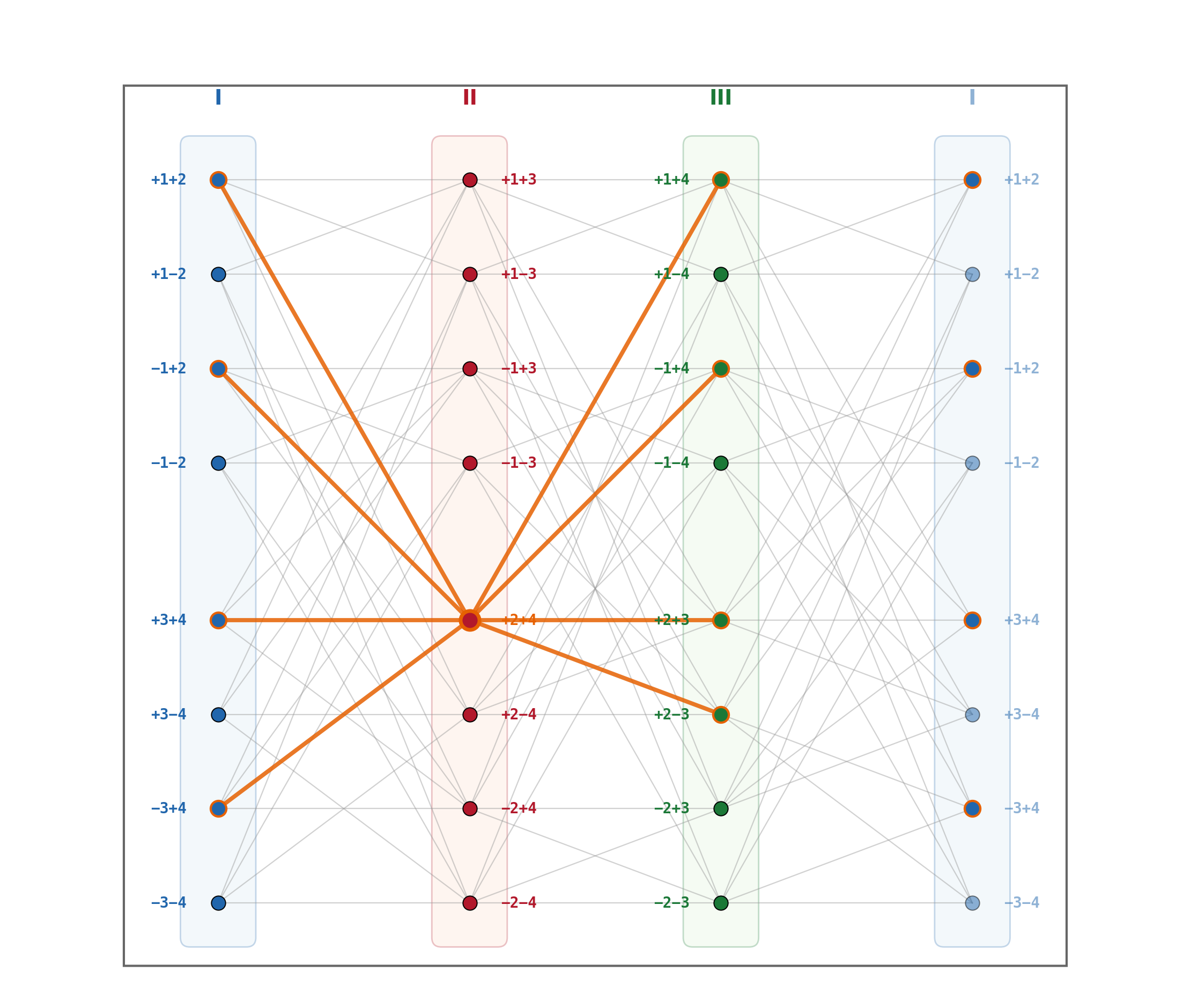}
\caption{The graph $\Gamma(F_4, 4)$.}
\label{fig:f4k4}
\end{figure}

\begin{example}[$\Gamma(F_4, 4)$]\label{ex:f4k4}
Consider three maximal strongly orthogonal subsets of $F_4$:
\begin{align*}
S_1 &= \{e_1 + e_3,\ e_1 - e_3,\ e_2 + e_4,\ e_2 - e_4\}, & v_1 &= \textstyle\sum S_1 = 2e_1 + 2e_2, \\
S_2 &= \{e_1 + e_2,\ e_1 - e_2,\ e_3 + e_4,\ e_3 - e_4\}, & v_2 &= \textstyle\sum S_2 = 2e_1 + 2e_3, \\
S_3 &= \{e_3 + e_1,\ e_3 - e_1,\ e_4 + e_2,\ e_4 - e_2\}, & v_3 &= \textstyle\sum S_3 = 2e_3 + 2e_4.
\end{align*}
It can be verified computationally that the sum of any SOS of maximal size is of the form $2(e_i - e_j)$; and that every such vector is the sum of an SOS. Hence $v_1 - v_2 = 2(e_2 - e_3)$, and so $v_1 \sim v_2$; but $v_1 - v_3 = 2e_1 + 2e_2 - 2e_3 - 2e_4$ and $v_1 \not\sim v_3$.

The vertices of $\Gamma(F_4, 4)$ are labelled by vectors with two non-zero entries. In Figure~\ref{fig:f4k4}, the vertex corresponding to $(0,2,0,2)$ is labelled $+2{+}4$ to indicate that the support is columns 2 and 4; and that the entries in those columns are positive. From Definition \ref{Graphdef}, two vertices are adjacent if and only if they share exactly one non-zero entry. The graph has independent sets labelled by the partitions of ${1,2,3,4}$ into parts of size 2. Part~I is repeated on both margins (identified) so that all edges appear as straight lines between adjacent columns. 

The full graph automorphism group is $\mathrm{Aut}(\Gamma(F_4, 4)) \cong \W(F_4) \cong S_3 \ltimes \W(D_4),$ of order $6 \times 192 = 1152$, verified using \texttt{nauty}~\cite{nauty}. The $S_3$ triality symmetry of the $D_4$ Dynkin diagram permutes the three parts. 
\end{example}

In fact, the vertices of $\Gamma(F_4, 4)$ are a copy of the $D_{4}$ root system scaled by a factor of $2$, and $\Gamma(F_4, 4) \cong \Gamma(D_4, 1)$. Later we will see further isomorphisms of this type.

\subsection{Lemmas on the graphs $\Gamma(R, k)$}

In this section we prove some straightforward Lemmas which will be used later to reduce the size of certain computations.

\begin{lemma}\label{lem:weyl}
For any root system $R$ and any positive integer $k$, the Weyl group $\W(R)$ acts as a group of automorphisms of $\Gamma(R, k)$.
\end{lemma}

\begin{proof}
Let $w \in \W(R)$. Since $w$ permutes the roots of $R$, it sends any $k$-element strongly orthogonal subset $S$ to another such subset $wS = \{w\alpha : \alpha \in S\}$. By linearity,
\[
\sum_{\alpha \in wS} \alpha = \sum_{\alpha \in S} w\alpha = w\!\left(\sum_{\alpha \in S} \alpha\right)\!,
\]
so $w$ maps $\mathcal{V}(R, k)$ to itself. Moreover, if two distinct SOS yield the same vertex sum, $\sum_{\alpha \in S} \alpha = \sum_{\alpha \in S'} \alpha$, then linearity gives $\sum_{\alpha \in wS} \alpha = \sum_{\alpha \in wS'} \alpha$, so $w$ is well-defined on $\mathcal{V}(R,k)$. For vertices $v_1, v_2 \in \mathcal{V}(R, k)$, linearity gives $w(v_1) - w(v_2) = w(v_1 - v_2)$, so $v_1 - v_2 \in \mathcal{V}(R, k)$ if and only if $w(v_1) - w(v_2) \in \mathcal{V}(R, k)$. Thus $w$ preserves the edge relation.
\end{proof}

For a simply-laced root system, the sum $v = \sum_{\alpha \in S} \alpha$ of any $S \in \mathrm{SOS}(R, k)$ satisfies $\|v\|^2 = 2k$, and all vertices of $\Gamma(R, k)$ lie on a sphere of radius $\sqrt{2k}$. When $k$ equals the rank of $R$, a stronger constraint applies.

\begin{lemma}\label{lem:mod8}
Let $R$ be a simply-laced root system of rank $\ell$, and suppose $k = \ell$. Then for any two vertices $v, w \in \mathcal{V}(R, k)$, we have $\|v - w\|^2 \equiv 0 \pmod{8}$.
\end{lemma}

\begin{proof}
Write $v = \sum_{i=1}^{\ell} \alpha_i$ and $w = \sum_{j=1}^{\ell} \beta_j$, where $\{\alpha_i\}$ and $\{\beta_j\}$ are $\ell$-element strongly orthogonal subsets. Since the $\alpha_i$ are pairwise orthogonal vectors of norm $2$ in a space of dimension $\ell$, they satisfy $\sum_{i=1}^{\ell} \langle \alpha_{i}, v\rangle v = 2v$ for any vector $v$; and similarly for $\{\beta_j\}$. By Axiom 4 of Definition \ref{RootDefn}, the entries of the matrix  $C$ given by $c_{ij} = \langle \alpha_i, \beta_j \rangle$ are integers. A computation shows that $C^\top C = 4 I_\ell$, and that $CC^\top = 4I_\ell$. 
So the rows of $C$ have squared entries summing to $4$: each row contains either a single non-zero entry $\pm 2$, or four non-zero entries $\pm 1$; in all cases the row sum is even.

Let $\mathbf{1} = (1, \ldots, 1)^\top$; and observe that the vector $s = \frac{1}{2}C\mathbf{1}$ has integer entries. Now
\[
\langle v, w \rangle = \mathbf{1}^\top C \mathbf{1} = 2 \sum_{i} s_{i},
\]
From the definitions of $s$ and $C$, we deduce that $\|s\|^2 = \tfrac{1}{4}\mathbf{1}^\top C^\top C \mathbf{1} = \tfrac{1}{4} \cdot 4\ell = \ell$. 
Since $n^2 \equiv n \pmod{2}$ for any integer $n$, we obtain $\sum_i s_i \equiv \sum_i s_i^2 = \ell \pmod{2}$. Therefore $\langle v, w \rangle = 2\sum_i s_i \equiv 2\ell \pmod{4}$, and
\[
\|v - w\|^2 = \|v\|^2 + \|w\|^2 - 2\langle v, w \rangle = 4\ell - 2\langle v, w \rangle \equiv 4\ell - 4\ell = 0 \pmod{8}. \qedhere
\]
\end{proof}

\begin{corollary}\label{cor:e7edgeless}
The graph $\Gamma(E_7, 7)$ has no edges.
\end{corollary}

\begin{proof}
By Lemma~\ref{lem:mod8}, every pairwise difference satisfies $\|v - w\|^2 \equiv 0 \pmod 8$. For $v - w$ to be a vertex requires $\|v - w\|^2 = 14$, but $14 \equiv 6 \pmod 8$, a contradiction.
\end{proof}

The following lemma allows efficient computation of the total number of maximal cliques in a vertex-transitive graph from a single neighbourhood.

\begin{lemma}\label{lem:nbr-count}
Let $\Gamma$ be a vertex-transitive graph on $n$ vertices in which every maximal clique has size $s$. Let $v$ be any vertex, and let $\Gamma[N(v)]$ denote the subgraph of $\Gamma$ induced on the neighbourhood of $v$. Then every maximal clique of $\Gamma[N(v)]$ has size $s-1$, and the total number of maximal cliques in $\Gamma$ is
\[
\frac{n \cdot c(v)}{s}\,,
\]
where $c(v)$ denotes the number of maximal cliques in $\Gamma[N(v)]$.
\end{lemma}

\begin{proof}
Every maximal clique $C$ in $\Gamma$ has size $s$, and for each $v \in C$ the set $C \setminus \{v\}$ is a clique of size $s-1$ in $\Gamma[N(v)]$. This clique is maximal in $\Gamma[N(v)]$: any extension by a vertex $u \in N(v) \setminus C$ would give a clique $C \cup \{u\}$ in $\Gamma$, contradicting the maximality of $C$. Conversely, if $C'$ is a maximal clique in $\Gamma[N(v)]$, then $C' \cup \{v\}$ is a clique in $\Gamma$; it is maximal since any extension by $u \notin N(v)$ is impossible, and any extension by $u \in N(v) \setminus C'$ contradicts maximality of $C'$ in $\Gamma[N(v)]$. This gives a bijection between maximal cliques of $\Gamma$ containing $v$ and maximal cliques of $\Gamma[N(v)]$, so $v$ lies in exactly $c(v)$ maximal cliques of $\Gamma$.

Since $\Gamma$ is vertex-transitive, every vertex lies in the same number $c(v)$ of maximal cliques. Counting incidences between vertices and maximal cliques: each of the $M$ cliques contributes $s$ incidences, and each of the $n$ vertices contributes $c(v)$, so $s \cdot M = n \cdot c(v)$.
\end{proof}

\begin{remark}\label{rem:two-orbit}
When $\Gamma$ has two vertex orbits $\mathcal{O}_1$ and $\mathcal{O}_2$ under $\W(R)$, with $|\mathcal{O}_i| = n_i$ and each vertex in $\mathcal{O}_i$ lying in $c_i$ maximum cliques of size $s$, a similar argument to that of Lemma \ref{lem:nbr-count} gives
\[
M = \frac{n_1 c_1 + n_2 c_2}{s}.
\]
This applies to $\Gamma(E_7, 4)$, and $\Gamma(E_8, 4)$, each of which splits into two Weyl orbits with distinct degrees (see Table~\ref{tab:properties}).
\end{remark}

We establish a closed-form degree formula for $k = 1$ below; a formula for $k > 1$ remains open. No originality is claimed in the second half of this proof; it is a well-known computation in the theory of root systems, included here for completeness. 

\begin{proposition}\label{prop:k1degree}
Let $R$ be an irreducible simply-laced root system of rank $\ell$ with Coxeter number $h$. Then $\Gamma(R, 1)$ is regular of degree $2(h - 2)$.
\end{proposition}

\begin{proof}
Since $R$ is simply-laced, all roots have the same norm $\|\alpha\|^2 = 2$, and $\W(R)$ acts transitively on $\mathcal{V}(R, 1)$, which is just the set of roots. Vertex-transitive graphs are regular.

The degree of vertex $\alpha$ is the number of roots $\beta$ such that $\alpha-\beta$ is a root; in a simply laced system, this is equal to the number of roots with $\langle \alpha, \beta \rangle = 1$. The sum $Q(x) = \sum_{\beta \in R} \langle x, \beta \rangle^2$ defines a positive-definite quadratic form on $V$ that is invariant under $\W(R)$. Since $R$ is irreducible, $\W(R)$ acts irreducibly on $V$, so by Schur's lemma $Q(x) = c\|x\|^2$ for some constant $c > 0$. Taking the trace over any orthonormal basis $\{f_i\}$ of $V$:
\[
\sum_{i=1}^\ell Q(f_i) = \sum_{\beta \in R} \|\beta\|^2 = 2|R|, \qquad \sum_{i=1}^\ell c\|f_i\|^2 = c\ell,
\]
so $c = 2|R|/\ell$ and $\sum_{\beta \in R} \langle x, \beta \rangle^2 = \frac{2|R|}{\ell}\|x\|^2$ for all $x \in V$.
Setting $x = \alpha$ and decomposing the left-hand side by the value of $\langle \alpha, \beta \rangle$:
\[
\underbrace{4 + 4}_{\beta = \pm\alpha} + \underbrace{d \cdot 1 + d \cdot 1}_{\langle \alpha, \beta \rangle = \pm 1} = \frac{2|R|}{\ell} \cdot 2,
\]
where $d = |\{\beta \in R : \langle \alpha, \beta \rangle = 1\}| = |\{\beta \in R : \langle \alpha, \beta \rangle = -1\}|$, the latter equality holding since $\beta \mapsto -\beta$ is a bijection. Solving gives $d = 2|R|/\ell - 4$, and since $|R| = \ell h$ for any irreducible root system \cite{Humphreys}, we obtain $d = 2(h - 2)$.
\end{proof}

Computations show that $\Gamma(R, k)$ is regular for all simply-laced exceptional $R$ and all $k$, with three exceptions: $(E_7, 3)$, $(E_7, 4)$, and $(E_8, 4)$. In each case the vertex set splits into exactly two Weyl orbits with distinct degrees. For $k = 4$, the smaller orbit is $\{2\alpha : \alpha \in R\}$---these vertices arise from strongly orthogonal subsets that ``pair up'' as $\{e_i + e_j,\, e_i - e_j,\, e_k + e_l,\, e_k - e_l\}$, whose sum is $2(e_i + e_k)$, twice a root. This orbit has $|R|$ elements (126 for $E_7$, 240 for $E_8$) and strictly higher degree than the generic orbit. For $\Gamma(E_7, 3)$, the smaller orbit consists of 56 isolated vertices.

\subsection{Computation of graphs} 

The graphs were constructed by generating all roots of $R$ using the explicit descriptions above, and finding all $k$-cliques in the ``strong orthogonality graph'' (vertices are roots, edges connect strongly orthogonal pairs) to construct $\mathrm{SOS}(R,k)$. An element of this set is a $k$-set of vectors, taking their sum gives $\mathcal{V}(R,k)$; and edges on this set are constructed by considering pairwise differences. The graph $\Gamma(E_8, 6)$ required computing $\binom{60480}{2} \approx 1.8 \times 10^9$ pair tests; simple but explicit memory management was required; we computed edges in batches and completed most tasks by streaming edges rather loading them all into memory simultaneously. Code was written in Python, calling nauty and cliquer (via SageMath) when required, \cite{nauty, cliquer}. The computational results described were established separately by the authors; the second author then used Claude Code (Anthropic's AI coding assistant) as a verification tool; which successfully recovered all computational results. The AI assistant developed relatively straightforward code rapidly, e.g. to compute the graphs $\Gamma(R, k)$, and to call tools such as cliquer to compute clique numbers. It suggested the key steps in the proof of Proposition 3.8, which was then refined by the authors. On the other hand it required multiple attempts and repeated corrections to implement Lemma \ref{lem:sf-perm}, producing plausible but invalid proofs and counts along the way. 

\subsection{Graph Parameters}

Table~\ref{tab:properties} presents the complete data. Here $|V|$ and $|E|$ are vertex and edge counts, $\delta$ and $\Delta$ are minimum and maximum degrees, and ``Comp.'' is the number of connected components.

\begin{table}[h]
\centering
\begin{tabular}{llrrrrrll}
\toprule
System & $k$ & $|V|$ & $|E|$ & min deg & max deg & CC & $\mathrm{Aut}(\Gamma)$ & Notes \\
\midrule
$G_2$ & 1 & 12 & 30 & 4 & 6 & 1 & $\W(G_2)$ & \\
$G_2$ & 2 & 6 & 6 & 2 & 2 & 1 & $\W(G_2)$ & regular \\
\midrule
$F_4$ & 1 & 48 & 408 & 14 & 20 & 1 & $\W(F_4)$ & \\
$F_4$ & 2 & 120 & 1200 & 20 & 20 & 1 & $\W(F_4)$ & regular \\
$F_4$ & 3 & 240 & 3552 & 26 & 32 & 1 & $\W(F_4)$ & \\
$F_4$ & 4 & 24 & 96 & 8 & 8 & 1 & $\W(F_4)$ & regular \\
\midrule
$E_6$ & 1 & 72 & 720 & 20 & 20 & 1 & $\W(E_6){:}2$ & regular \\
$E_6$ & 2 & 270 & 4590 & 34 & 34 & 1 & $\W(E_6){:}2$ & regular \\
$E_6$ & 3 & 720 & 26640 & 74 & 74 & 1 & $\W(E_6){:}2$ & regular \\
$E_6$ & 4 & 72 & 720 & 20 & 20 & 1 & $\W(E_6){:}2$ & $\cong \Gamma(E_6, 1)$ \\
\midrule
$E_7$ & 1 & 126 & 2016 & 32 & 32 & 1 & $\W(E_7)$ & regular \\
$E_7$ & 2 & 756 & 37800 & 100 & 100 & 1 & $\W(E_7)$ & regular \\
$E_7$ & 3 & 2072 & 183456 & 0 & 182 & 57 & $S_{56} \times \W(E_7)$ & \\
$E_7$ & 4 & 4158 & 582624 & 272 & 544 & 1 & $\W(E_7)$ & \\
$E_7$ & 5 & 7560 & 1572480 & 416 & 416 & 1 & $\W(E_7)$ & regular \\
$E_7$ & 6 & 10080 & 1844640 & 366 & 366 & 1 & $\W(E_7)$ & regular \\
$E_7$ & 7 & 576 & 0 & 0 & 0 & 576 & $S_{576}$ & edgeless \\
\midrule
$E_8$ & 1 & 240 & 6720 & 56 & 56 & 1 & $\W(E_8)$ & regular \\
$E_8$ & 2 & 2160 & 302400 & 280 & 280 & 1 & $\W(E_8)$ & regular \\
$E_8$ & 3 & 6720 & 1821120 & 542 & 542 & 1 & $\W(E_8)$ & regular \\
$E_8$ & 4 & 17520 & 10409280 & 1176 & 2072 & 1 & $\W(E_8) \leq \mathrm{Aut}$ & \\
$E_8$ & 5 & 30240 & 22014720 & 1456 & 1456 & 1 & $\W(E_8) \leq \mathrm{Aut}$ & regular \\
$E_8$ & 6 & 60480 & 81950400 & 2710 & 2710 & 1 & $\W(E_8) \leq \mathrm{Aut}$ & regular \\
$E_8$ & 7 & 69120 & 67737600 & 1960 & 1960 & 1 & $\W(E_8) \leq \mathrm{Aut}$ & regular \\
$E_8$ & 8 & 2160 & 302400 & 280 & 280 & 1 & $\W(E_8)$ & $\cong \Gamma(E_8, 2)$ \\
\bottomrule
\end{tabular}
\caption{Parameters of $\Gamma(R, k)$ graphs. CC is number of connected components. $\W(R){:}2$ denotes the Weyl group extended by a diagram automorphism. $\W(R) \leq \mathrm{Aut}$ indicates that $\W(R)$ acts by automorphisms but the full automorphism group has not been determined.}
\label{tab:properties}
\end{table}

Among all graphs $\Gamma(R,k)$ for exceptional $R$, exactly two are disconnected: $\Gamma(E_7, 3)$ with 57 components (including 56 isolated vertices), and $\Gamma(E_7, 7)$ which is edgeless; the latter is explained by Corollary \ref{cor:e7edgeless}. For simply-laced $R$, the graphs $\Gamma(R, 1)$ are 1-skeletons of Gosset polytopes~\cite{Coxeter}: $\Gamma(E_6, 1)$, $\Gamma(E_7, 1)$, and $\Gamma(E_8, 1)$ are the graphs of the polytopes $2_{22}$, $1_{32}$, and $4_{21}$ respectively. The inner product values $\langle \alpha, \beta \rangle \in \{-2, -1, 0, 1, 2\}$ between distinct roots define a 4-class association scheme on $R$, of which $\Gamma(R, 1)$ is the first relation graph. The containment $\W(R) \leq \mathrm{Aut}(\Gamma(R,k))$ is established by Lemma~\ref{lem:weyl}; the automorphism group was determined to be no larger using nauty for all but the four largest cases. Table~\ref{tab:properties} also records two isomorphisms: $\Gamma(E_6, 1) \cong \Gamma(E_6, 4)$ and $\Gamma(E_8, 2) \cong \Gamma(E_8, 8)$. These arise from a scaling relationship between vertex sets, which we now explain.

\begin{proposition}\label{prop:duality}
The following isomorphisms hold, each given by the scaling map $v \mapsto 2v$:
\begin{enumerate}
    \item $\mathcal{V}(E_6, 4) = 2 \cdot \mathcal{V}(E_6, 1)$, giving $\Gamma(E_6, 1) \cong \Gamma(E_6, 4)$.
    \item $\mathcal{V}(E_8, 8) = 2 \cdot \mathcal{V}(E_8, 2)$, giving $\Gamma(E_8, 2) \cong \Gamma(E_8, 8)$.
\end{enumerate}
\end{proposition}

\begin{proof}
In each case, once the bijection between vertex sets is established, the map $v \mapsto 2v$ is automatically a graph isomorphism: $v_1 - v_2 \in \mathcal{V}$ if and only if $2(v_1 - v_2) \in 2\mathcal{V}$. 

Let $R$ be a simply-laced root system with maximum SOS size $m$, and let $S = \{\alpha_1, \ldots, \alpha_m\}$ be a maximal SOS. Set $\gamma = \frac{1}{2}\sum_{i=1}^m \alpha_i$. Since the $\alpha_i$ are pairwise orthogonal with $\|\alpha_i\|^2 = 2$, we have $\|\gamma\|^2 = m/2$ and $\langle \gamma, \alpha_i \rangle = 1$ for each $i$.

By \cite{Agaoka}, $\W(R)$ acts transitively on maximal SOS in any irreducible root system. It therefore suffices to verify the claim for a single representative maximal SOS; surjectivity then follows from the transitivity of $\W(R)$ on the relevant target vertex set.

\textbf{Case $E_6$ ($m = 4$):} Here $\|\gamma\|^2 = 2$; we show $\gamma$ is always an $E_6$ root. Take $S_0 = \{e_2 + e_3,\, e_2 - e_3,\, e_4 + e_5,\, e_4 - e_5\}$. Then $\gamma_0 = e_2 + e_4$, which is an $E_6$ root. Since $\W(E_6)$ acts transitively on maximal SOS, every centre is of the form $w\gamma_0$ for some $w \in \W(E_6)$. But $w$ permutes the roots of $E_6$, so every centre is a root: $\mathcal{V}(E_6, 4) \subseteq 2\mathcal{V}(E_6, 1)$. For the reverse inclusion, $\W(E_6)$ acts transitively on roots, so every root $\alpha$ has the form $\alpha = w\gamma_0$; then $2\alpha$ is the sum of the maximal SOS $wS_0$, giving $2\alpha \in \mathcal{V}(E_6, 4)$.

\textbf{Case $E_8$ ($m = 8$):} Here $\|\gamma\|^2 = 4$, so $\gamma$ is not a root. Take $S_0 = \{e_1 \pm e_2,\, e_3 \pm e_4,\, e_5 \pm e_6,\, e_7 \pm e_8\}$. The centre is $\gamma_0 = e_1 + e_3 + e_5 + e_7 = (e_1 + e_3) + (e_5 + e_7)$. Since $e_1 + e_3$ and $e_5 + e_7$ are strongly orthogonal $E_8$ roots, $\gamma_0 \in \mathcal{V}(E_8, 2)$. Transitivity on maximal SOS gives $\mathcal{V}(E_8, 8) \subseteq 2\mathcal{V}(E_8, 2)$. For surjectivity, we show that $\W(E_8)$ acts transitively on $\mathcal{V}(E_8, 2)$. Since $E_8$ is simply laced, $\W(E_8)$ acts transitively on roots; the stabiliser of a root $\alpha$ has order $|\W(E_8)|/|E_8| = 696{,}729{,}600/240 = 2{,}903{,}040 = |\W(E_7)|$, so $\mathrm{Stab}(\alpha) = \W(E_7)$, which acts transitively on the roots of $E_8$ orthogonal to $\alpha$. It follows that $\W(E_8)$ acts transitively on ordered pairs of strongly orthogonal roots, so every element of $\mathcal{V}(E_8, 2)$ is a Weyl translate of $\gamma_0$, and hence a centre. Thus $2\mathcal{V}(E_8, 2) \subseteq \mathcal{V}(E_8, 8)$.
\end{proof}

We conclude this section with a tabulation of maximal clique sizes for the graphs $\Gamma(R, k)$ with $R$ a root system of exceptional type.

\begin{table}[h]
\centering
\begin{tabular}{lcccccccc}
\toprule
System & $k=1$ & $k=2$ & $k=3$ & $k=4$ & $k=5$ & $k=6$ & $k=7$ & $k=8$ \\
\midrule
$G_2$ & 3 & 2 & & & & & & \\
$F_4$ & 7 & 3 & 3 & 3 & & & & \\
$E_6$ & 5 & 3 & 5 & 5 & & & & \\
$E_7$ & 7 & 6 & 5 & 7 & 5 & 4 & 1 & \\
$E_8$ & 8 & 8 & 8 & 8 & 8 & 8 & 8 & 8 \\
\bottomrule
\end{tabular}
\caption{Clique numbers of $\Gamma(R, k)$.}
\label{tab:clique}
\end{table}

\section{Erd\H{o}s--Ko--Rado Properties}\label{sec:ekr}

The classical EKR theorem can be stated graph-theoretically: the \emph{Kneser graph} $K(n,k)$ has vertices $\binom{[n]}{k}$ with edges between \emph{disjoint} pairs; its complement (the \emph{intersection graph}) has edges between \emph{intersecting} pairs. Cliques in the intersection graph are intersecting families, and EKR bounds the clique number.

With Gashi, the authors explored EKR properties of the type $A_\ell$ root system. Roots are vectors in $\mathbb{R}^{\ell + 1}$ with two non-zero entries, one positive and one negative; roots are strongly orthogonal if and only if they are disjoint. The graphs $\Gamma(A_\ell, k)$ encode when sums of matchings differ by another such sum. A \emph{sunflower} is a collection of vertices labelled by vectors where every pair shares the same common part. That is: when the vectors are tabulated, each column must contain only non-zero entries (the \textit{core}) or at most one non-zero entry (the \textit{petals}). The main result of the previous work was an EKR type theorem, showing that all cliques of maximum size are sunflowers when $\ell$ is much larger than $k$.

\begin{theorem}[\cite{BGO}]\label{BGOmain}
For $\ell > k \cdot 4^k$, every maximal clique in $\Gamma(A_\ell, k)$ is a sunflower.
\end{theorem}

This definition can be applied to root systems of exceptional type without difficulty. 

\begin{definition}\label{def:sunflower}
For $v \in \mathbb{R}^n$, the support is $\mathrm{supp}(v) = \{i : v_i \neq 0\}$. A clique $\{v_1, \ldots, v_p\}$ in $\Gamma(R,k)$ is a \emph{sunflower} if there exists a set $C$ such that $\mathrm{supp}(v_i) \cap \mathrm{supp}(v_j) = C$ for all $i \neq j$.
\end{definition}

Our convention is that the core must be non-empty; though the petals might be empty: a sunflower might consist of vectors with no zero entries. 

\begin{example}[Non-sunflower clique in $\Gamma(E_8, 3)$]\label{ex:sunflower}
In $\Gamma(E_8, 3)$, vertices are sums of three pairwise orthogonal roots. Consider:
\begin{align*}
v_1 &= (\phantom{-}1, -1, \phantom{-}1, \phantom{-}0, -1, -1, \phantom{-}1, \phantom{-}0), & \mathrm{supp}(v_1) &= \{1,2,3,5,6,7\}, \\
v_2 &= (-1, -1, \phantom{-}1, -1, \phantom{-}0, -1, \phantom{-}1, \phantom{-}0), & \mathrm{supp}(v_2) &= \{1,2,3,4,6,7\}, \\
v_3 &= (\phantom{-}1, \phantom{-}0, \phantom{-}1, -1, \phantom{-}1, -1, \phantom{-}1, \phantom{-}0), & \mathrm{supp}(v_3) &= \{1,3,4,5,6,7\}.
\end{align*}
These form a clique in $\Gamma(E_8, 3)$: each difference (e.g.\ $v_1 - v_2 = (2, 0, 0, 1, -1, 0, 0, 0)$) is a vertex. The pairwise support intersections are distinct, $\mathrm{supp}(v_1) \cap \mathrm{supp}(v_2) = \{1,2,3,6,7\}$ and $\mathrm{supp}(v_1) \cap \mathrm{supp}(v_3) = \{1,3,5,6,7\},$ so this is not a sunflower. 
\end{example}

\subsection{Sunflower Proportions}

We computed all maximum cliques and classified them by the sunflower property.\footnote{Alternative definitions were also tested. A value-based definition (requiring constant values on core coordinates) and a sign-based definition (requiring constant signs) both yield strictly lower sunflower proportions. For $E_7$, $k=3$: a support-based definition gives 3.2\% of maximum cliques satisfying the sunflower property, while neither the value-based nor sign-based definitions give any sunflowers at all.} The sunflower property is coordinate-dependent: a Weyl group element $w$ maps cliques to cliques but may change coordinate supports, so $w$ does not in general preserve whether a clique is a sunflower. However, the subgroup of coordinate \emph{permutation} matrices acts by relabelling supports, preserving all pairwise intersections, and hence the sunflower property. This observation reduces the sunflower count to a manageable computation.

\begin{lemma}\label{lem:sf-perm}
Let $\Gamma$ be a graph on which a group $G$ acts by automorphisms, with maximum cliques of size $\omega$. Let $P \leq G$ be the subgroup of coordinate permutation matrices, and write $\mathrm{sf}(v)$ for the number of maximum sunflower cliques containing~$v$. If $\mathcal{O}_1, \ldots, \mathcal{O}_r$ are the $P$-orbits on $V(\Gamma)$ with representatives $v_1, \ldots, v_r$, the total number of maximum sunflower cliques is
\[
\frac{1}{\omega} \sum_{i=1}^{r} |\mathcal{O}_i| \cdot \mathrm{sf}(v_i)\,.
\]
\end{lemma}

\begin{proof}
Consider the matrix with $\omega$ rows formed from the vertex labels of a maximum clique. A \textit{core column} has all entries non-zero; and a \textit{petal column} has one non-zero entry. In a sunflower, all non-zero columns are either core columns or petal columns. Since $\sigma \in P$ acts by permuting columns, it maps core columns to core columns; petal columns to petal columns and hence sunflowers to sunflowers. If $v_2 = \sigma v_1$ with $\sigma \in P \leq G$, then $\sigma$ is an automorphism of $\Gamma$ mapping the cliques through $v_1$ bijectively to those through~$v_2$, preserving the sunflower property, so $\mathrm{sf}(v_1) = \mathrm{sf}(v_2)$.

Count the pairs $(v, C)$ where $C$ is a maximum clique and $v \in C$ in two ways. First: by summing over vertices, it is $\sum_{v} \cdot \mathrm{sf}(v)$; since these counts are constant on $P$-orbits, this is $\sum_{i=1}^{r} |\mathcal{O}_i| \cdot \mathrm{sf}(v_i)$. Second: if the number of maximum cliques is $\mathcal{C}$, the count is $\omega \mathcal{C}$. Dividing both sides by $\omega$ gives the result. 
\end{proof}

For the total number of maximum cliques, Lemma~\ref{lem:nbr-count} applies directly in the vertex-transitive case; for the three non-regular instances $(E_7, 3)$, $(E_7, 4)$, and $(E_8, 4)$ that each split into two Weyl orbits, the formula of Remark~\ref{rem:two-orbit} is used. For the root systems in the $E_8$ embedding, the permutation subgroup of the Weyl group has an explicit description. In $E_8$, every coordinate permutation is a Weyl group element---the transposition $(i,j)$ arises as the reflection in $e_i - e_j$---giving $P \cong S_8$. For $E_7$ with constraint $x_1 + x_8 = 0$, the permutations preserving this constraint are $S_2 \times S_6$, where $S_2$ swaps coordinates $1$ and $8$ while $S_6$ permutes $\{2, \ldots, 7\}$; this gives $r = 7$ orbits on $\mathcal{V}(E_7, 1)$ and between $16$ and $53$ orbits for higher~$k$. For $E_6$ with constraints $x_1 + x_7 = 0$ and $x_1 + x_8 = 0$, the permutation subgroup is $S_5$, permuting $\{2, \ldots, 6\}$.

\begin{table}[ht]
\centering
\begin{minipage}{0.48\textwidth}
\centering
\begin{tabular}{llrrr}
\toprule
System & $k$ & Cliques & Sunfl. & \% \\
\midrule
$G_2$ & 1 & 20 & 0 & 0.0 \\
$G_2$ & 2 & 6 & 6 & 100.0 \\
\midrule
$F_4$ & 1 & 24 & 0 & 0.0 \\
$F_4$ & 2 & 1,152 & 192 & 16.7 \\
$F_4$ & 3 & 4,992 & 896 & 17.9 \\
$F_4$ & 4 & 96 & 64 & 66.7 \\
\midrule
$E_7$ & 1 & 576 & 0 & 0.0 \\
$E_7$ & 2 & 120,960 & 0 & 0.0 \\
$E_7$ & 3 & 483,840 & 15,360 & 3.2 \\
$E_7$ & 4 & 1,021,824 & 104,448 & 10.2 \\
$E_7$ & 5 & 7,547,904 & 119,808 & 1.6 \\
$E_7$ & 6 & 4,838,400 & 122,880 & 2.5 \\
\bottomrule
\end{tabular}
\end{minipage}\hfill
\begin{minipage}{0.48\textwidth}
\centering
\begin{tabular}{llrrr}
\toprule
System & $k$ & Cliques & Sunfl. & \% \\
\midrule
\midrule
$E_6$ & 1 & 432 & 32 & 7.4 \\
$E_6$ & 2 & 4,320 & 0 & 0.0 \\
$E_6$ & 3 & 17,280 & 1,280 & 7.4 \\
$E_6$ & 4 & 432 & 32 & 7.4 \\

\midrule
$E_8$ & 1 & 17,280 & 128 & 0.7 \\
$E_8$ & 2 & 4,665,600 & 30,720 & 0.7 \\
$E_8$ & 3 & 38,707,200 & 286,720 & 0.7 \\
$E_8$ & 4 & 635,316,480 & 4,705,536 & 0.7 \\
$E_8$ & 5 & 679,311,360 & 5,031,936 & 0.7 \\
$E_8$ & 6 & 10,450,944,000 & 68,812,800 & 0.7 \\
$E_8$ & 7 & 1,194,393,600 & 8,847,360 & 0.7 \\
$E_8$ & 8 & 4,665,600 & 30,720 & 0.7 \\
\bottomrule
\end{tabular}
\end{minipage}
\caption{Maximum cliques and sunflower proportions across exceptional root systems. Counts are of cliques of maximum size~$\omega$.}
\label{tab:sunflower}
\end{table}

The $E_8$ and $E_7$ clique counts were computed via Lemma~\ref{lem:nbr-count} using \texttt{cliquer} \cite{cliquer} to enumerate cliques in $\Gamma[N(v)]$; sunflower counts for $E_7$ and $E_8$ used Lemma~\ref{lem:sf-perm}. For non-regular cases ($(E_7,3)$, $(E_7,4)$, $(E_8,4)$) the per-orbit weighted formula was used. All maximum cliques in $\Gamma(E_8, k)$ have size $s = 8$; for $E_7$, mixed maximal clique sizes occur. $\Gamma(F_4, 1)$ also has 336 maximal cliques of size~5 (the clique number is $\omega = 7$); of these, 16 are sunflowers. Likewise, for $E_7$, the maximal clique sizes are not uniform: e.g.\ $\Gamma(E_7,4)$ has 1,021,824 cliques of maximum size~7 but also 3,870,720 maximal cliques of size~6. All counts in the table refer to cliques of maximum size only. The $E_7$ sunflower counts were independently verified by direct enumeration for $k \leq 3$ and by the permutation class method of Lemma~\ref{lem:sf-perm} for all~$k$.

The sunflower property (Definition~\ref{def:sunflower}) depends on coordinate supports, and hence on the choice of basis. All results in Table~\ref{tab:sunflower} use the coordinates described in Section~\ref{sec:rootsystems}: the standard orthonormal coordinates for $G_2$ and $F_4$, and the 8-dimensional $E_8$ embedding for $E_6$, $E_7$, and $E_8$. For $E_6$, one could instead use the embedding in $\mathbb{R}^{6}$ spanned by linear combinations of the rows of the Cartan matrix (see \cite{Humphreys}), which changes the supports and hence the sunflower classification. The graph $\Gamma(E_6, k)$ and its clique structure are basis-independent; only the sunflower classification changes.

\begin{table}[h]
\centering
\begin{tabular}{lcccc}
\toprule
Basis & $k=1$ & $k=2$ & $k=3$ & $k=4$ \\
\midrule
$E_8$ embedding (8D) & 7.4\% & 0.0\% & 7.4\% & 7.4\% \\
Cartan basis (6D) & 0.0\% & 0.7\% & 0.0\% & 0.0\% \\
\bottomrule
\end{tabular}
\caption{Sunflower proportions for $E_6$ in two different coordinate systems. }
\label{tab:basis_dependence}
\end{table}

\section{Open Questions and Concluding Remarks}\label{sec:conclusion}
Generalising work on SOS cliques in the type $A_{\ell}$ root systems led us to define, construct and analyse a series of graphs related to the root systems of exceptional type. Our main theoretical results relate properties of these graphs to properties of the underlying root systems; unexpected isomorphisms between graphs are explained by the geometry of the underlying SOS cliques. 

In the simply-laced exceptional types $E_6$, $E_7$, and $E_8$, sunflower cliques are rare, comprising at most 11\% of maximum cliques, compared to 100\% in type $A$ for large rank. For the non-simply-laced types $G_2$ and $F_4$, sunflower proportions are higher and more variable. This suggests that the condition $\ell \gg k$ in Theorem \ref{BGOmain} is likely to be necessary; it remains open to determine the precise behaviour of sunflower and non-sunflower cliques in the $A_\ell$ system when $k$ is proportional to $\ell$.

Our work suggests three questions. 

\begin{question}
For which root systems $R$ does an isomorphism $\Gamma(R_1, k) \cong \Gamma(R_2, m)$ hold? Proposition~\ref{prop:duality} establishes graph isomorphisms $\Gamma(E_6, 1) \cong \Gamma(E_6, 4)$ and $\Gamma(E_8, 2) \cong \Gamma(E_8, 8)$ via the scaling map $v \mapsto 2v$; and Example \ref{ex:f4k4} shows that $\Gamma(F_4, 4) \cong \Gamma(D_4, 1)$.  
\end{question}

\begin{question}
Is there a representation-theoretic interpretation of $\Gamma(R,k)$? Strongly orthogonal roots appear in the construction of discrete series representations and in Kostant's cascade construction, \cite{HarishChandra, KostantCascade}. Do the graph theoretic properties of $\Gamma(R,k)$ have a representation-theoretic interpretation?
\end{question}

\begin{question}
Does a result analogous to Theorem \ref{BGOmain} hold for types $B_\ell$, $C_\ell$, $D_\ell$? If so, what is the optimal relationship between $\ell$ and $k$ under which the EKR property holds? 
\end{question}

\end{document}